\documentclass[12pt]{amsart}
\usepackage{amssymb,amsmath,amsthm,mathrsfs}
\usepackage[usenames]{color}
\usepackage{hyperref}
\usepackage[all]{xy}
\usepackage{xypic}
\usepackage{graphicx}
\usepackage{amscd}
\usepackage{enumerate}

\baselineskip 18pt \textwidth 16cm  \theoremstyle{plain}

\newtheorem{theorem}{Theorem}[section]
\newtheorem{proposition}[theorem]{Proposition}
\newtheorem{corollary}[theorem]{Corollary}
\newtheorem{lemma}[theorem]{Lemma}
\newtheorem{definition}[theorem]{Definition}
\newtheorem{notation}[theorem]{Notation}

\newtheorem{example}[theorem]{Example}
\newtheorem{remark}[theorem]{Remark}

\newcommand{\BP}{Brenner Property }

\DeclareMathOperator{\SL}{SL}
\DeclareMathOperator{\GL}{GL}

\DeclareMathOperator{\rank}{rank}

\DeclareMathOperator{\covol}{covol}
\DeclareMathOperator{\aut}{Aut}
\DeclareMathOperator{\Aut}{Aut}
\DeclareMathOperator{\kernel}{Ker}
\DeclareMathOperator{\Ad}{Ad}
\DeclareMathOperator{\linspan}{span}
\DeclareMathOperator{\Sp}{Sp}
\DeclareMathOperator{\Gal}{Gal}

\def\Z{\Bbb Z}

\def\Q{\Bbb Q}

\def\N{\Bbb N}

\def\la{\langle}
\def\ra{\rangle}
\def\nom{\triangleleft}
\DeclareMathOperator\tor{T}

\title{First order rigidity
 of non-uniform higher rank arithmetic groups}

\author{Nir Avni}

\address{Nir Avni,
Department of Mathematics, Northwestern University, Evanston, IL 60201, USA.}
\email{avni.nir@gmail.com}
\urladdr{\url{http://math.northwestern.edu/\~nir}}

\author{Alexander Lubotzky}
\address{Alexander Lubotzky,
Department of mathematics, the Hebrew University, Jerusalem, Israel.}

\email{alex.lubotzky@mail.huji.ac.il}
\urladdr{\url{http://www.ma.huji.ac.il/~alexlub}}

\author{Chen Meiri}
\address{Chen Meiri,
Department of Mathematics, Technion, Haifa, Israel.}
\email{chenm@technion.ac.il}

\begin{document}

\maketitle

\begin{flushright}\emph{ In Memory of Daniel G. Mostow }
\end{flushright}

\begin{abstract} If $\Gamma$ is an irreducible non-uniform higher-rank
characteristic zero 
 arithmetic lattice (for example $\SL_n(\mathbb{Z})$, $n \geq 3$) and $\Lambda$ is a finitely generated group that is elementarily equivalent to $\Gamma$, then $\Lambda$ is isomorphic to $\Gamma$.
\end{abstract}

\section{Introduction}

In this article, we state and prove a new rigidity result for irreducible non-uniform higher-rank arithmetic lattices. This class includes the groups $\SL_n(\mathbb{Z})$ for $n \geq 3$ and $\SL_n(\mathbb{Z}[1/p])$ for $n,p  \geq 2$.

We recall the definitions. A lattice in a locally compact, second countable group $\mathbf{G}$ is a discrete subgroup $\Gamma \subset \mathbf{G}$ such that there is a fundamental domain with finite Haar measure for the translation action of $\Gamma$ on $\mathbf{G}$. A lattice is called uniform if $\mathbf{G}/\Gamma$ is compact, and non-uniform otherwise. We say that $\Gamma$ is irreducible if the image of $\Gamma$ is dense in any quotient of $\mathbf{G}$ by a  non-compact normal subgroup of $\mathbf{G}$. 

In this paper by a semisimple group we mean a locally compact group $\mathbf{G}$ of the form $\prod_{i=1}^r G_i(F_i)$, where $F_i$ are local fields of characteristic zero  and $G_i$ are connected simple algebraic groups defined over $F_i$ and $G_i(F_i)$ is non-compact for every $1 \le i \le r$. We say that a semisimple group $\mathbf{G}$ has higher-rank if $\sum \rank_{F_i}G_i \geq 2$ and has low-rank otherwise.  
A group which is an irreducible lattice in a semisimple group of higher-rank is called a higher-rank lattice. 
By Mostow's strong rigidity (see Theorem A and page 9 in \cite{Mos}), a group cannot be an irreducible lattice in both a semisimple higher-rank group and a semisimple low-rank group, so being an irreducible lattice in a higher-rank group is a property of $\Gamma$. For example, $\SL_n(\mathbb{Z})$ is an irreducible non-uniform lattice in $\SL_n(\mathbb{R})$, $n \geq 2$; it is a higher-rank lattice if $n\geq 3$, while $\SL_n(\mathbb{Z}[1/p])$ is an irreducible non-uniform higher-rank lattice in $\SL_n(\mathbb{R}) \times \SL_n(\mathbb{Q}_p)$ for any $n,p \geq 2$. 

By Margulis' Arithmeticity Theorem, every irreducible higher-rank lattice is arithmetic, in the following sense: Let $k$ be a number field with ring of integers $O$, let $S$ be a finite set of places of $k$, containing all the archimedean ones, and let $O_S:=\left\{ x\in k \mid \left( \forall v\notin S \right) v(x) \geq 0 \right\}$ be the ring of $S$-integers. Let $G$ be a 
connected
group scheme over $O_S$, and let $G_k$ be the corresponding algebraic group over $k$. Assume that  $G_k$ is absolutely simple and simply connected. Any
group which is abstractly commensurable to such $G(O_S)$ is called an arithmetic group. 
Borel and Harish-Chandra \cite{BHC} proved that the image (under the diagonal embedding) of  
$G(O_S)$ in  $\prod_{v \in S}G(K_v)$ is an irreducible lattice and so every arithmetic group is commensurable to an irreducible lattice in some semisimple group.


Irreducible higher-rank lattices have many remarkable properties. For example, Margulis's Superrigidity Theorem roughly says that $\Gamma$ (as abstract group) determines $\mathbf{G}$ and the embedding $\Gamma \hookrightarrow \mathbf{G}$ up to automorphisms of $\mathbf{G}$ (see Definition \ref{defn:superrigid} for the accurate statement). Another amazing rigidity result for these groups is the following (for a quick formulation, we assume that $\Gamma$ is non-uniform): If $\Lambda$ is any finitely generated group which is quasi-isometric to $\Gamma$ (i.e., the Cayley graph of $\Lambda$ is quasi-isometric to that of $\Gamma$), then, up to finite index and finite normal subgroups, $\Gamma$ and $\Lambda$ are isomorphic, see  \cite{Far} and the reference therein. 

The main goal of this paper is to show a new rigidity phenomenon for higher rank arithmetic groups. For the formulation, we need the following definitions:

\begin{definition} Two groups are said to be \emph{elementarily equivalent} if every first order sentence in the language of groups that holds in one also holds in the other.
\end{definition} 

Elementary equivalence is fairly weak equivalence relation: every infinite group has an equivalent group of any infinite cardinality. From a group-theoretic perspective, it is reasonable to restrict the discussion to finitely generated groups. Luckily, characteristic zero arithmetic groups are always finitely generated (in fact, finitely presented).

\begin{definition} We say that a finitely generated group $\Gamma$ is \emph{first order rigid} if every finitely generated group that is elementarily equivalent to $\Gamma$ is isomorphic to $\Gamma$.
\end{definition} 

Finitely generated abelian groups are first order rigid. Nilpotent groups need not be first order rigid, but the elementary equivalence classes of any nilpotent group is finite (see Remark \ref{rem:ee.finite}). In general, elementary equivalence classes can be infinite. The celebrated work of Sela \cite{Sel02} (see also \cite{KM}) shows that all non-abelian free groups are elementarily equivalent, and are also equivalent to the fundamental groups of compact surfaces of genera at least two (free groups and fundamental groups of surfaces are all arithmetic groups, but not of higher-rank). 

Our main result says that the situation for higher-rank arithmetic groups is very different.
Recall that two groups are said to be abstractly commensurable if they contain isomorphic finite index subgroups. 

\begin{theorem} \label{thm:main} Any group which is abstractly commensurable to an irreducible non-uniform higher-rank lattice is first order rigid.
\end{theorem}

\begin{remark}\label{rem:connected}  First order rigidity is, in general, not preserved under abstract commensurability, see \S\ref{sec:commensurable}.
\end{remark} 

\begin{remark} 
\emph{(1)} 
Theorem \ref{thm:main} stands in a sharp contrast to lattices in low-rank groups: \begin{enumerate}
\item[{(a)}] By \cite{Sel09}, all torsion-free lattices in $\SL_2(\mathbb{R})$ are elementarily equivalent. 
\item[{(b)}] By \cite[Theorem 7.6]{Sel09}, if $\Gamma$ is torsion-free uniform lattice in a rank-one group, then $\Gamma$ is elementarily equivalent to $\Gamma * F_n$ for all $n \geq 1$.
\end{enumerate}
\noindent 
\emph{2)} By \cite[Proposition 7.1]{Sel09}, two non-isomorphic uniform torsion-free lattices in rank-one groups other than $\SL_2(\mathbb{R})$ are never elementarily equivalent. We do not know what happens for non-uniform lattices.
\end{remark}

\begin{remark} We can prove that many irreducible non-uniform higher-rank arithmetic lattices in positive characteristics are first order rigid and we 
 believe that all of them are. While we speculate that higher rank uniform lattices are also first order rigid, we do not have a single example where we can prove it. \end{remark}

\begin{remark} In a sequel to this article we show that for many higher-rank lattices $\Gamma$, there is a single statement $\phi_\Gamma$ such that, if $\Lambda$ is a finitely generated group, then $\Lambda$ satisfies $\phi_\Gamma$ if and only if $\Lambda$ is isomorphic to $\Gamma$.
\end{remark}

The paper is organized as follows: Section \ref{sec:preliminaries} contains some preliminaries, including the crucial definitions of a prime group and the Brenner property. In the same section, we also show that $\SL_n(\mathbb{Z})$ is prime and has an element satisfying the Brenner property. In Section \ref{sec:main.proof}, we prove that a prime group with a finite center that has an element with the Brenner property is first-order rigid. This finishes the proof of rigidity for $\SL_n(\mathbb{Z})$. In Section \ref{sec:prime} we show that superrigid arithmetic groups are prime and in Section \ref{sec:Raghunathan} we show that irreducible higher-rank non-uniform lattices have elements with the Brenner property, finishing the proof of Theorem \ref{thm:main}.  Finally, in Section \ref{sec:commensurable} we show that first-order rigidity is, in general, not preserved under commensurability.
\\ \\
\emph{This article is dedicated to the memory Daniel G. Mostow who is  the founding father of modern rigidity. Dan was a role model and inspiration for us, professionally and personally. }
\\ \\
{\bf Acknowledgement:} The authors are grateful to Zlil Sela for fruitful conversations and insights which improved the original proof. 
We are also thankful to Goulnara Arzhantseva, Andre Nies, Andrei Rapinchuk and Tyakal Nanjundiah Venkataramana for pointing out to us several background references. 
The first author was partially support by NSF grant no. DMS-1303205 and BSF grant no. 2012247. 
The second author was partially support by ERC, NSF and BSF. The third author was partially supported by ISF grant no. 662/15
and BSF grant no. 2014099.

\section{Preliminaries} \label{sec:preliminaries}

The following is a theorem of Malcev. For completeness, we include a proof.

\begin{proposition}[\cite{Mal}] \label{prop:res.finite} If $\Lambda$ is a group that is elementarily equivalent to a linear group, then $\Lambda$ is linear. If, in addition, $\Lambda$ is finitely generated, then $\Lambda$ is residually finite.
\end{proposition} 

\begin{proof} Let $\Phi \subset \GL_n(k)$ be a linear group which is elementary equivalent to $\Lambda$. Enumerate the elements of $\Lambda$ as $\Lambda=\left\{ \lambda_\alpha \right\}_{\alpha \in A}$, and enumerate all relations that hold between the $\lambda_\alpha$s by $\left\{ r_\beta (\lambda_\alpha) \right\}_{\beta \in B}$. Let $\mathcal{L}$ be the first-order language of rings together with constants $c^\alpha_{i,j}$, for $\alpha \in A$ and $1 \leq i,j\leq n$. Consider the theory $\mathcal{T}$ consisting of the following statements: \begin{enumerate} 
\item The axioms of fields.
\item The statements $\det(c^{\alpha}_{i,j}) \neq 0$, for all $\alpha \in A$.
\item The statements $(c^{\alpha}_{i,j}) \neq (c^{\beta}_{i,j})$, for all $\alpha,\beta \in A$.
\item The statements $r_\beta ((c^{\alpha}_{i,j}))=1$, for all $\beta \in B$.
\end{enumerate} 

If $\mathcal{S}$ is a finite subset of $\mathcal{T}$, then there is a finite set $A_0 \subset A$ such that $\mathcal{S}$ involves only $c^{\alpha}_{i,j}$ for $\alpha \in A_0$. Since the elements $\lambda_\alpha$, $\alpha \in A_0$ satisfy all relations $r_\beta$ that involve only them, we get that there are elements $(g^{\alpha}_{i,j})\in \Phi \subset \GL_n(k)$ that satisfy $\mathcal{S}$. In particular, $\mathcal{S}$ is consistent. By the Compactness Theorem, there is a model $K$ of $\mathcal{T}$. This $K$ must be a field, and the map $\lambda ^ \alpha \mapsto (c^ \alpha _{i,j})$ is an embedding $\Lambda \hookrightarrow \GL_n(K)$.

Finally, we show that the second claim follows from the first. If $\Lambda$ is finitely generated and linear, then there is a finitely generated ring $A$ such that $\Lambda \subset \GL_n(A)$. Since a finitely generated ring is residually finite, it follows that $\Lambda$ is residually finite.
\end{proof} 

\begin{definition} A homomorphism $f: \Gamma \rightarrow \Lambda$ is called an \emph{ elementary embedding} if, for every first order formula $\phi(\vec{x})$ with $n$ free variables and every $\vec{a}\in\Gamma^n$, the statement $\phi(\vec{a})$ holds in $\Gamma$ if and only if $\phi(f(\vec{a}))$ holds in $\Lambda$.
\end{definition} 

\begin{definition} We say that a group $\Gamma$ is \emph{prime} if, for every group $\Lambda$ that is elementary equivalent to $\Gamma$, there is an elementary embedding $\Gamma \hookrightarrow \Lambda$.
\end{definition} 

The following is proved by Oger and Sabbagh:

\begin{theorem}[\cite{OS}] \label{thm:prime.criterion} Let $\Gamma$ be a finitely generated group. The following are equivalent: \begin{enumerate} 
\item $\Gamma$ is prime.
\item There is a generating tuple $\vec{g}\in \Gamma^n$ and a formula $\phi(\vec{x})$ such that, for any $n$-tuple $\vec{h}\in \Gamma^n$, the statement $\phi(\vec{h})$ holds in $\Gamma$ if and only if $\vec{h}$ is in the $\Aut(\Gamma)$ orbit of $\vec{g}$.
\end{enumerate} 
\end{theorem}

\begin{example} \label{expl:SLn} $\SL_n(\mathbb{Z})$, $n \geq 3$, is prime: We use the following consequence of superrigidity: Any endomorphism of $\SL_n(\mathbb{Z})$ is either trivial or an automorphism.

Fix a finite presentation $\langle g_1,\ldots,g_a \mid r_1,\ldots,r_b \rangle$ of $\SL_n(\mathbb{Z})$, and let $\phi(x_1,\ldots,x_a)$ be the formula
\[
\phi(\vec{x})=(x_1 \neq 1) \wedge \bigwedge_{j=1}^b (r_j(\vec{x})=1).
\]
If $\vec{h}\in \left( \SL_n(\mathbb{Z}) \right) ^a$ and $\phi(\vec{h})$ holds, then the map $g_i \mapsto h_i$ extends to a non-trivial endomorphism of $\Gamma$, so it must be an automorphism, so $\vec{h}$ is a generating tuple.
\end{example}

\begin{notation}\label{not:setpower} For a set $S \subset \Gamma$ and $n \geq 1$, let $[S]^n=\left\{ g_1 \cdots g_n \mid g_i\in S\cup \left\{ 1 \right\} \right\}$.
\end{notation}

\begin{definition} We say that an element $g\in \Gamma$ has the \emph{\BP} if there exists a constant $D \geq 1$ for which the following statement hold:
\[
\text{For every $h \in \Gamma$, if $|[h^\Gamma \cup (h ^{-1})^\Gamma]^D | > D$ then $[h^\Gamma \cup (h ^{-1})^\Gamma]^D \cap Z(C_\Gamma(g)) \neq \left\{ 1 \right\}$. }
\]
\end{definition} 

\begin{remark}\label{rem:BP} 
Let $h \in \Gamma$ and denote $S=h^\Gamma \cup (h ^{-1})^\Gamma$.
We claim that  if $|[S]^D | \le  D$ for some $D \ge 1$ then $[S]^D$ is the normal subgroup of $\Gamma$ generated by $h$. and in particular $[S]^D=[S]^C$
for every $C \ge D$. The proof of the claim is by induction on $D$. Note that $|[S]|^1 \le  1$ if and only if $h=e$ so the claim holds for $D=1$. Assume that $h \ne e$ and that $|[S]^D | \le  D$ for some $D \ge 2$.  Define $T_m:=[S]^m$
and $t_m:= |T_m| $ for every  $1 \le m \le D$. Note that 
$ T_1 \subseteq T_2 \subseteq \cdots\subseteq T_D$
and $2 \le t_1 \le t_2 \le \cdots \le t_D=D$. Thus, there exists $1 \le m < D$ for which $t_{m}=t_{m+1}$ and $T_m=T_{m+1}$. It easily follows that $T_m$ is the normal subgroup generated  
by $h$. 

Note that the claim implies that  if $h \in \Gamma$ is not contained in any finite normal subgroup then $|[S]^D | >  D$ for every $D \ge 1$. 
\end{remark}

Let $e_{i,j}\in \SL_n(\mathbb{Z})$ be the elementary matrix with 1s on the diagonal and entry $(i,j)$ and zero elsewhere.

\begin{lemma}[\cite{Bre}] \label{lem:Brenner} Let $n \ge 3$. Then $e_{1,n}$ has the \BP in $\SL_n(\Z)$.
\end{lemma}

\begin{proof} Denote $\Gamma=\SL_n(\Z)$. 
Since the center of $G$ is finite it is enough to show that there is a constant $C$ such that for every $h \in \Gamma \setminus Z(\Gamma)$, 
$[h^\Gamma \cup (h ^{-1})^\Gamma]^C \cap Z(C_\Gamma(e_{1,n})) \neq \left\{ 1 \right\}$.

Let $h \in \Gamma \setminus Z(\Gamma)$ and define $S:=h^\Gamma \cup (h ^{-1})^\Gamma$. 
For every $k \ge 1$, $[S]^k$ is a symmetric normal subset. Thus, if $t \in [S]^k$ and $q \in \Gamma$ then  $t^\Gamma \cup (t^{-1})^\Gamma
\subseteq [S]^k$ and $[t,q]:=tqt^{-1}q^{-1} \in [S]^{2k}$.
For a matrix $t \in \SL_n(\Q)$ let $V_t:=\{v \in \Q^n\ \mid tv=v\}$. 
As $\SL_n(\Z)=\la e_{i,j}\mid 1 \le i \ne j \le n \ra$, there exists $1 \le r \ne s \le n$ such that $1 \ne h^*:=[e_{r,s},h] \in [S]^2$. Since  $\dim(V_{e_{r,s}})=\dim(V_{he_{r,s}^{-1}h^{-1}})=n-1$, 
we get  $n-2 \le \dim(V_{h^*}) \le n-1$.
By the structure theorem of finitely generated abelian groups, there exist $0\neq A \in M_{n-2 \time 2 }(\Z)$ and $B \in \SL_{2 }(\Z)$
such that $h$ is conjugate in $\SL_n(\Z)$ to  $$h^{**}=\left(\begin{array}{cc} I_{n-2} & A \\
0 & B \end{array}\right) \in [S]^2.$$
By considering the cases $B=\pm I_2$ and $B \ne \pm I_2$ separately,  it is easy to see that there exist
$0 \ne A',A'' \in M_{n-2 \time 2 }(\Z)$ and  $B' \in \SL_{2 }(\Z)$  such that
$$1 \ne h^{***}:=\left
[\left(\begin{array}{cc} I_{n-2} & A \\
0 & B \end{array}\right),
\left(\begin{array}{cc} I_{n-2} & A' \\
0 & B' \end{array}\right)\right]=
\left(\begin{array}{cc} I_{n-2} & A'' \\
0 & I_2 \end{array}\right) \in [S]^4.$$
If $h^{***} $ differs from the identity matrix only in the last column then $h^{***}$ is conjugate to $e_{1,n}^k$ for some $k \ne 0$. Otherwise,
$h^{****}:=[h^{**},e_{n-1,n}] \in [S]^8$ is a non-identity matrix which differs from the identity matrix only in the last column.
\end{proof}

\section{Primeness and Brenner property imply first order rigidity} \label{sec:main.proof}

\begin{proposition}\label{prop:prime.gp} Let $\Gamma$ be a finitely generated group with finite center
which has a maximal finite normal subgroup N (i.e., any finite normal subgroup is contained in N). 
Suppose that $\Lambda$ is a finitely generated group, $i:\Gamma \rightarrow \Lambda$ is an elementary embedding, and $b\in \Gamma$. Then
\begin{enumerate}
\item \label{cond:prime.gp.1} $i(Z(\Gamma))=Z(\Lambda)$.
\item \label{cond:prime.gp.2} $i(Z(C_{\Gamma}(b)))=Z(C_{\Lambda}(i(b)))\cap i(\Gamma)$.
\item \label{cond:prime.gp.4} $i(N)$ is a maximal finite normal subgroup of $\Lambda$. In particular, any non-trivial finite normal subgroup of $\Lambda$ is contained in $i(\Gamma)$.
\item \label{cond:prime.gp.3} If $\Lambda$ is finitely presented or $\Gamma$ is linear then, for every $t\in \Lambda \smallsetminus \left\{ 1 \right\}$, there is $\Delta_t \nom \Lambda$ such that $\Lambda=\Delta_t \mathbb{o} i(\Gamma)$ and $t \not \in \Delta_t$. In particular, $Z(C_\Lambda(i(b)))$ is the direct sum of $Z(C_\Lambda(i(b)))\cap \Delta_t$ and $i(Z(C_\Gamma(b)))$.
\end{enumerate}
\end{proposition}

\begin{proof} Identify $\Gamma$ with its image in $\Lambda$. Let $\vec{g}$ be a generating $n$-tuple of $\Gamma$.  
Let $Z(\Gamma)=\left\{a_1,\ldots,a_m\right\}$. Let $\xi(x_1,\ldots,x_m)$ be the formula saying that every central element is one of the $x_i$s. Since $\xi(a_1,\ldots,a_m)$ holds in $\Gamma$, it holds in $\Lambda$. This implies \eqref{cond:prime.gp.1}.

For every word $w(\vec{x})$, let $\nu_w(\vec{x},y)$ be the first order formula saying that $w(\vec{x})$ is in the center of the centralizer of $y$. If $\nu_w(\vec{g},b)$ holds in $\Gamma$, then it holds in $\Delta$. This implies \eqref{cond:prime.gp.2}.

 Note that the elements $h_1,\ldots,h_k $ belong to a finite normal subgroup of size at most $C$ if and only if 
$[S_{h_1,\ldots,h_k}]^C \le C$ where $S_{h_1,\ldots,h_k}$ is the union of the  conjugacy classes of $h_1,\ldots,h_k,h_1^{-1},\ldots,h_k^{-1}$ and $[S]^C$ is as in Notation  \ref{not:setpower}.  
Let $D$ be the size of $N$. For every $k\ge 1$ and $C \ge D$ let $\psi_{k,C}$ be a first order sentence which states that for every elements $h_1,\ldots,h_k$ if
$[S_{h_1,\ldots,h_k}]^C \le C$ then $[S_{h_1,\ldots,h_k}]^D \le D$. Then $\psi_{k,C}$ holds in $\Gamma$ and therefore it also holds in $\Lambda$. 
It follows that a finite normal subgroup of $\Lambda$ is of size at most $D$. In a similar manner to the proof of part  \eqref{cond:prime.gp.1} we see that since $i$ is an elementary embedding 
then $i(N)$ is a finite normal subgroup of $\Lambda$ of size $D$. If $M$ is any other finite normal subgroup of $\Lambda$ then $i(N)M$ is also a finite normal subgroup of $\Lambda$. 
Since $|i(N)M|\le D$ and $|i(N)|=D$ then $M \le i(N)$. This implies \eqref{cond:prime.gp.4}.

Finally we prove \eqref{cond:prime.gp.3}. In order to prove the first part of \eqref{cond:prime.gp.3}  it is enough to find an epimorphism $\varphi_t:\Lambda \rightarrow \Gamma$ whose restriction to $\Gamma$ is the identity map and such that $\varphi_t(t)\neq 1$ (since, in this case, $\Lambda=\ker(\varphi_t) \mathbb{o} \Gamma$).  
Let $\la y_1,\ldots,y_m \mid r_1,\ldots \ra$ be a presentation for $\Lambda$, and let $\vec{h}:=(h_1,\ldots,h_m)$ be a generating $m$-tuple  of $\Lambda$ corresponding to this presentation.
If $\Lambda$ is finitely presented then there are only finitely many relations $r_1,\ldots,r_s$ and for every $m$ matrices $l_1,\ldots,l_m \in \GL_d(F)$ which satisfy these relations,
 the map $h_i \mapsto l_i$ extends to a homomorphism from $\Lambda$ to $\Gamma$ . If $\Gamma$ is linear,
Hilbert's basis theorem implies that there exists a number $s$ such that any $m$ matrices $l_1,\ldots,l_m \in \GL_d(F)$ satisfying the relations $r_1,\ldots,r_s$ also satisfy the rest of the relations $r_i$, and, in particular, the map $h_i \mapsto l_i$ extends to a homomorphism. 
Let $w_1(\vec{x}),\ldots,w_n(\vec{x})$ be words such that $g_i=w_i(\vec{h})$ for every $1 \le i \le n$, and let $u(\vec{x})$ be a word such that $t=u(\vec{h})$. Let $\eta(y_1,\ldots,y_m,x_1,\ldots,x_n)$ be the first order formula which is the conjunction of
\begin{enumerate}
\item $\vec{y}$ satisfies $r_1,\ldots,r_s$.
\item $\bigwedge_{1 \le i \le n}w_i(\vec{y})=x_i$.
\item $u(\vec{y})\neq 1$.
\end{enumerate}
The tuple $\vec{h}$ is a testment that the formula $(\exists \vec{y}) \eta(\vec{y},\vec{g})$ holds in $\Lambda$. Hence, this formula also holds in $\Gamma$. Let $\vec{k}\in \Gamma^m$ be such that $\eta(\vec{k},\vec{g})$ holds. By the first part of $\eta$, the map $h_i \mapsto k_i$ extends to a homomorphism $\varphi_t: \Lambda \rightarrow \Gamma$. By the second part of the definition of $\eta$, $\varphi_t(g_i)=\varphi_t(w_i(\vec{h}))=w_i(\vec{k})=g_i$ for every $1 \le i \le n$, so the restriction of $\varphi_t$ to $\Gamma$ is the identity map. By the third part of the definition of $\eta$, $\varphi_t(t)=\varphi_t(u(h))=u(\vec{k})\neq 1$. The second part of \eqref{cond:prime.gp.3} follows from the first part and \eqref{cond:prime.gp.2}.

\end{proof}

\begin{remark} In this paper, the requirement that $t \not \in \Delta_t$ in  part \eqref{cond:prime.gp.3} of Proposition \ref{prop:prime.gp} does not play any role. 
This requirement becomes important when dealing with the positive characteristic case and it is included here for future reference. 
\end{remark}

\begin{theorem} \label{thm:f.o.r.criterion} Let $F$ be a field and let $\Gamma \le \GL_d(F)$ be a finitely generated prime group
with a Brenner element $b$.
Assume that:
\begin{enumerate}
\item $\Gamma$ has a finite center. 
\item $\Gamma$ has a maximal  finite normal subgroup N (i.e., any finite normal subgroup is contained in N). 
\item There exists  $k$ such that the minimal number of generators of every finitely generated 
subgroup of $Z(C_\Gamma(b))$ is at most $k$.
\end{enumerate} 
Then $\Gamma$ is first order rigid. 
\end{theorem}

\begin{proof} Let $\Lambda$ be a finitely generated group that is elementarily equivalent to $\Gamma$. Since $\Gamma$ is prime, there is an elementary embedding $i:\Gamma \rightarrow \Lambda$. As before, we identify $\Gamma$ with $i(\Gamma)$. By Proposition \ref{prop:prime.gp}\eqref{cond:prime.gp.3} (with any non-trivial $t$), there is a subgroup $\Delta \nom \Lambda$ such that $\Lambda=\Delta \mathbb{o} \Gamma$. We will show that $\Delta=1$. 

If there is a non-trivial element in $\Delta$, the Brenner property of $b$, Remark \ref{rem:BP}, part \eqref{cond:prime.gp.4} of Proposition \ref{prop:prime.gp} and the normality of $\Delta$ imply that there is a non-trivial element in $\Delta\cap Z(C_{\Lambda}(b))$. Hence, it is enough to prove that $\Delta\cap Z(C_{\Lambda}(b))=\left\{1\right\}$.

For an abelian group $\Phi$, denote the set of $m$-powers in $\Phi$ by $P_m(\Phi)$, and note that this is a subgroup. Since every finitely generated subgroup of $Z(C_{\Gamma}(b))$ is generated
by at most $k$ elements, the group $Z(C_{\Gamma}(b)) / P_m(Z(C_{\Gamma}(b)))$ is finite. Let $d_m$ be its size. There is a first order formula $\nu_m(x)$ that says that the quotient of the center of the centralizer of $x$ by the collection of $m$-th powers of the center of the centralizer of $x$ has size $d_m$. Since $\nu_m(b)$ holds in $\Gamma$, it also holds in $\Lambda$. Hence, $|Z(C_{\Lambda}(b)) / P_m(Z(C_{\Lambda}(b)))|=|Z(C_{\Gamma}(b)) / P_m(Z(C_{\Gamma}(b)))|$. Proposition \ref{prop:prime.gp}\eqref{cond:prime.gp.3} implies that, for every $m$, $Z(C_\Lambda(b))\cap \Delta = P_m(Z(C_\Lambda(b))\cap \Delta)$. 
Hence, $Z(C_\Lambda(b))\cap \Delta$ is divisible. 
By Proposition \ref{prop:res.finite}, $\Lambda$ is linear. Since $\Lambda$ is finitely generated, it is residually finite. It follows that $Z(C_\Lambda(b))\cap \Delta$ is a divisible and residually finite group, a contradiction.
\end{proof} 

Combining Theorem \ref{thm:f.o.r.criterion}, Lemma \ref{lem:Brenner}, Example \ref{expl:SLn}, and noting that $Z(C_{\SL_n(\mathbb{Z})}(e_{1,n}))$ is the cyclic group generated by $e_{1,n}$, we get
\begin{corollary} If $n \geq 3$, $\SL_n(\mathbb{Z})$ is first order rigid.
\end{corollary}

\section{Superrigid lattices are prime} \label{sec:prime}

In this section, we prove that superrigid lattices are prime. Recall our notation that $G,H,\ldots$ denote algebraic groups and $\mathbf{G},\mathbf{H},\ldots$ denote locally compact groups.

\begin{definition} \label{defn:superrigid} A subgroup  $\Gamma$ of a locally compact group $\mathbf{G}$ is called \emph{superrigid} if, for any simple adjoint algebraic group $H$ defined over a local field $L$, any homomorphism from $\Gamma$ to $H(L)$, whose image is unbounded and Zariski dense, extends to a homomorphism from $\mathbf{G}$ to $H(L)$.
\end{definition} 

There are many examples of superrigid subgroups:

\begin{example} \,\ 
\begin{enumerate} 
\item By Theorem (2) in page 2 of \cite{Mar}, irreducible lattices in higher-rank semisimple groups are superrigid.
\item By \cite{Cor} and \cite{GS}, lattices in $\Sp(n,1)$ and in $F_4^{(-20)}$ are superrigid.
\item In \cite{BL} there were given examples of groups which are superrigid but not lattices. 
\end{enumerate} 
\end{example} 

Recall that a semisimple group is a locally compact group $\mathbf{G}=\prod_{i=1}^r G_i(F_i)$, where $F_i$ are local fields of characteristic zero  and $G_i$ are connected simple algebraic groups defined over $F_i$ and $G_i(F_i)$ is non-compact for every $1 \le i \le r$. 
The purpose of this section is to prove the following:

\begin{theorem} \label{thm:sr.prime} Let $\Phi$ be a group which is abstractly commensurable  to an irreducible lattice in higher-rank (characeterstic zero) semisimple group. 
Then $\Phi$ is prime.
\end{theorem}

Some preparation is needed for the proof of Theorem \ref{thm:sr.prime} which is given below. 

\begin{definition} Let $f:\mathbf{G} \rightarrow \mathbf{H}$ be a homomorphism between two locally compact groups. We say that $f$ is \emph{locally measure preserving} if there is a neighborhood $1\in U \subset \mathbf{G}$ such that $f|_U:U \rightarrow f(U)$ is a measure preserving homeomorphism. 
\end{definition}

Note that, if $f: \mathbf{G}\rightarrow \mathbf{H}$ is locally measure preserving and the restriction of $f$ to $\Omega \subset \mathbf{G}$ is one-to-one, then $f|_\Omega$ is measure preserving.

\begin{example} If $G,H$ are semisimple algebraic groups defined  over a local field $K$ and $f:G \rightarrow H$ is a central isogeny (i.e., $f$ is surjective and $\ker(f)$ is a finite subgroup of the center of $G$) with invertible derivative, then, up to normalization of the Haar measures by constants, the map $f:G(K) \rightarrow H(K)$ is locally measure preserving.
\end{example} 

Suppose that $f: \mathbf{G} \rightarrow \mathbf{H}$ is locally measure preserving and onto, and let $\Lambda \subset \mathbf{G}$ be a discrete subgroup such that $\Lambda \supset \kernel(f)$. If $\Omega \subset \mathbf{G}$ is a fundamental domain for $\Lambda$ in $\mathbf{G}$, then $f(\Omega)$ is a fundamental domain for $f(\Lambda)$ and $f|_\Omega : \Omega \rightarrow f(\Omega)$ is one-to-one. It follows that the covolume of $\Lambda$ in $\mathbf{G}$ is equal to the covolume of $f(\Lambda)$ in $\mathbf{H}$.

We will use a theorem of Borel and Tits. In the following statement, if $G$ is an algebraic group over a field $F$, we denote  by $G^+$ the subgroup of $G(F)$ generated by the subgroups $U(F)$, where $U$ ranges over the unipotent radicals of parabolic subgroups of $G$. If $\beta : F \rightarrow F'$ is a homomorphism of fields, we denote the base change of $G$ by $\beta$ by ${}^\beta G$.

\begin{theorem}[\cite{BT-abstract}, Theorem A] Let $F,F'$ be fields. Let $G$ and $G'$ be absolutely simple connected algebraic groups over $F$ and $F'$ respectively. Assume that $G'$ is adjoint and $G^+$ is Zariski dense in $G$. Let $f:G(F) \rightarrow G'(F')$ be a homomorphism with Zariski-dense image. Then there is a field homomorphism $\beta : F \rightarrow F'$, an $F'$-isogeny with invertible derivative $\phi:{}^{\beta}G \rightarrow G'$, and a homomorphism $\gamma:G(F) \rightarrow Z(G'(F'))$ such that $f(g)=\gamma(g) \phi(\beta(g))$.
\end{theorem}

\begin{remark} By the solution to the Knesser--Tits conjecture, if $F$ is either a local or a global field and $G$ is $F$-isotropic, then $G^+$ is Zariski dense in $G$. This implies also that $Z(G(F))=Z(G)$. We will only apply the theorem under the assumption that $F$ and $F'$ are either local or global, so the condition on $G^+$ is always satisfied and $Z(G'(F'))=Z(G')=1$.
\end{remark}

\begin{theorem} \label{thm:rigidity} Let $G$ be a connected absolutely simple group over a number field $k$, let $S$ be a finite set of valuations, containing all archimedean ones. Assume that every finite index subgroup   of $G(O_S)$  is  superrigid in $\prod_{v\in S} G(k_v)$. Let $\Gamma$ be a finite index subgroup of $G(O_S)$ and let  $\rho : \Gamma \rightarrow G(O_S)$ be a homomorphism with infinite image. Then 
\begin{enumerate}
\item\label{cond:rigidity.1} $\ker(\rho)$ is finite.
\item\label{cond:rigidity.2} If $\rho$ is injective, then $[G(O_S): \Gamma]=[G(O_S): \rho(\Gamma)]$. In particular, if $\rho(\Gamma) \subset \Gamma$, then $\rho$ is an automorphism.
\end{enumerate}
\end{theorem} 

\begin{proof} Note first that it suffices to prove (1) and (2) for some finite index subgroup of $\Gamma$, so we can replace $\Gamma$ with a finite index subgroup whenever it is needed. 
 \begin{enumerate}
\item[{\bf Step 1}]  Since $Z(G(O_S))$ is finite, by passing to a finite index subgroup of $\Gamma$ we may assume that $\Gamma \cap Z(G(O_S))=1$. We can also assume that, for any $v\in S$, $\Gamma$ is unbounded in the valuation $v$ (otherwise, after passing to a finite-index subgroup, $\Gamma \subset G(O_{S\smallsetminus \left\{ v \right\}})$). Denote $\mathbf{G}=\prod_{v\in S} G(k_v)$, and let $\delta : G(k) \rightarrow \mathbf{G}$ be the diagonal embedding.

\item[{\bf Step 2}] Let $H=\overline{\rho(\Gamma)}^Z$ be the Zariski closure of the image of $\Gamma$ and let $H^0$ be the connected component of identity. Since the image of $\rho$ is infinite, $H^0$ is not trivial. Replace $\Gamma$ with $\Gamma \cap H^0$ (and still call it $\Gamma$). There is $v\in S$ and a non-trivial adjoint $k$-factor $q:H^0 \rightarrow K$ such that $q(\rho(\Gamma) )$ is unbounded in the valuation $v$.

{\bf Proof:} Assume the contrary. Let $q:H^0 \rightarrow K$ be an adjoint factor defined over $k$, and choose a $k$-embedding $K \hookrightarrow \GL_n$. Since $q$ is defined over $k$, the group $q \circ \rho(\Gamma)$ is commensurable to a subgroup of $K(k)\cap \GL_n(O_S)$, so it is discrete in $\prod_{v\in S} K(k_v)$. Being pre-compact, $q \circ \rho(\Gamma)$ is finite. Since $q \circ \rho(\Gamma)$ is also Zariski dense in the connected group $K$, it follows that $K$ is trivial. Since this holds for every adjoint factor $K$, $H^0$ is solvable. Thus, $H^0$ has an infinite abelianization, so $\Gamma$ has 
a finite index subgroup with 
an infinite abelianization, a contradiction to superrigidity.

\item[{\bf Step 3}] $G=H=H^0$ and $K=G^{ad}$. I.e., $\rho(\Gamma)$ is Zariski dense.

{\bf Proof:} Since $q \circ \rho(\Lambda)$ is Zariski dense in $K$ and unbounded in $K(k_v)$, Superrigidity implies that $q \circ \rho$ extends to a map $f_v:\mathbf{G} \rightarrow K(k_v)$, so, in particular, there is $w\in S$ and a non-trivial map $f_{w,v}:G(k_w) \rightarrow K(k_v)$. By Borel--Tits, there is a field homomorphism $\beta=\beta_{w,v}: k_w \rightarrow k_v$ and a non-trivial algebraic homomorphism $\phi: {}^{\beta} G \rightarrow K$ such that $f_{w,v}$ is the composition of $\beta: G(k_w) \rightarrow {}^\beta G(k_v)$ and $\phi$. Since ${}^\beta G$ is simple, $\dim K \geq \dim {}^\beta G = \dim G \geq \dim H^0 \geq \dim K$, so $H^0$ is open in $G$. Since $G$ is connected, we get that $G=H^0=H$ and $K=G^{ad}$. 

\item[{\bf Step 4}] $\ker(\rho)$ is finite.

{\bf Proof:} In the last step we showed that $f_{w,v}(G(k_w))$ is Zariski dense in $K$. Since $K$ has trivial center, the image under $f_v$ of any other factor of $\mathbf{G}$ is trivial. It follows that $f_v$ is the composition of the projection $\mathbf{G} \rightarrow G(k_w)$ and $f_{w,v}$. Hence, $\rho$ is the composition of the embedding $\Gamma \rightarrow G(k_w)$ and $f_{w,v}$. By Borel--Tits, $f_{w,v}$ is a composition of a field homomorphism, which is necessarily injective, and a non-trivial central isogeny, so the kernel of $f_{w,v}$ is finite.

\item[{\bf Step 5}] Let $f_v:\mathbf{G} \rightarrow G^{ad}(k_v)$ be the map constructed in Step 3. Then $f_v(\delta(G(k))) \subset G^{ad}(k)$.

{\bf Proof:} Denote the algebraic closure of $k_w$ by $\overline{k_w}$. Let $g\in G(k)$, and assume that $f_v(\delta(g))\in G^{ad}(k_w)\smallsetminus G^{ad}(k) \subset G^{ad}(\overline{k_w})\smallsetminus G^{ad}(k)$. Then there is a field automorphism $\sigma\in \Gal(\overline{k_w}/k)$ such that $\sigma(f_v(\delta(g)))\neq f_v(\delta(g))$. Denote the conjugation by an element $h$ by $c_h$. Since $g\in G(k)$, there is a finite-index subgroup $\Lambda \subset \Gamma$ such that $c_g(\Lambda) \subset \Gamma$. It follows that $c_{f_v(\delta(g))}(f_v(\delta(\Lambda))) \subset f_v(\delta(\Gamma))=q\circ \rho(\Gamma) \subset G^{ad}(k)$. We get that, for each $h\in f_v(\delta(\Lambda))$, $c_{f_v(\delta(g))}(h)=c_{\sigma(f_v(\delta(g))}(h)$, meaning that $\sigma(f_v(\delta(g))) \left( f_v(\delta(g)) \right)^{-1}\neq 1$ commutes with $f_v(\delta(\Lambda))$. Since $f_v(\delta(\Lambda))$ has finite index in $f_v(\delta(\Gamma))$, it is Zariski dense. Since $G^{ad}$ has trivial center, we get a contradiction.

\item[{\bf Step 6}] Let $f=f_v \circ \delta:G(k) \rightarrow G^{ad}(k)$. By Borel--Tits, there is a field endomorphism $\alpha : k \rightarrow k$ and a central isogeny $\psi: {}^\alpha G \rightarrow G^{ad}$ such that $f=\psi \circ \alpha$. Since the characteristic of $k$ is zero, $\alpha$ is an automorphism. The automorphism $\alpha$ defines a bijection on the set of valuations of $k$ by $(\alpha(w))(x)=w(\alpha ^{-1}(x))$. We claim that $\alpha(S)=S$.

{\bf Proof:} 
Note that by our assumption $w \in S$ iff $G(O_S)$ is unbounded in $G(k_v)$.
Let $w\in S$. By our assumptions, $G(O_S)$ is unbounded in the $w$ valuation, so $\alpha(G(O_S))$ is unbounded in the $\alpha(w)$ valuation. Since $\psi$ is a central isogeny, this implies that $\psi(\alpha(G(O_S)))$ is unbounded in the $\alpha(w)$ valuation. Since $\psi(\alpha(G(O_S)))$ is commensurable to $q (\rho (\Gamma))$ and $q(\rho(\Gamma))\cap G^{ad}(O_S)$ has finite index in $q(\rho(\Gamma))$, it follows that $G^{ad}(O_S)$ is unbounded in the $\alpha(w)$ valuation, so $\alpha(w)\in S$.

\item[{\bf Step 7}] Let $q:G \rightarrow G^{ad}$ be the quotient by the center from before. Let 
\[
\mathbf{q} : \mathbf{G}=\prod_{w\in S} G(k_w) \rightarrow \prod_{w\in S} G^{ad}(k_w)= \prod_{w\in S} G(k_w)/Z(G(k_w))=\mathbf{G} / Z(\mathbf{G})
\]
be the map induced by $q$. Then the composition $\Gamma \stackrel{\rho}{\rightarrow} G(O_S) \stackrel{\delta}{\rightarrow} \mathbf{G} \stackrel{\mathbf{q}}{\rightarrow} \mathbf{G}/Z(\mathbf{G})$ extends to a locally measure preserving map $\mathbf{h}: \mathbf{G} \rightarrow \mathbf{G} / Z(\mathbf{G})$ (i.e., $\mathbf{h} \circ \delta = \mathbf{q} \circ \delta \circ \rho$), whose kernel is $Z(\mathbf{G})$.

{\bf Proof:} For every $w\in S$, the map $f:G(k) \rightarrow G^{ad}(k)$ is uniformly continuous if we put the $w$-topology on $G(k)$ and the $\alpha(w)$-topology on $G^{ad}(k)$, so it extends to a continuous map $h_w: G(k_w) \rightarrow G^{ad}(k_{\alpha(w)})$. Let $\mathbf{h}:\mathbf{G} \rightarrow \prod_{w\in S} G^{ad}(k_w) = \mathbf{G}/Z(\mathbf{G})$ be the product map. Each $h_w$ is a composition of an isomorphism and a central isogeny, so it is locally measure preserving. It is easy to see that $\mathbf{h}$ extends $\mathbf{q} \circ \delta \circ \rho$.

\item[{\bf Step 8}] We have $[G(O_S): \Gamma]=[G(O_S): \rho(\Gamma)]$.

{\bf Proof:} By Step 7, $\mathbf{q}(\delta(\rho(\Gamma)))=\mathbf{h}(\delta(\Gamma))$. We have
\[
\covol_{\mathbf{G}/Z(\mathbf{G})}(\mathbf{h}(\delta(\Gamma)))=\covol_{\mathbf{G}/Z(\mathbf{G})}(\mathbf{h}(\delta(\Gamma) Z(\mathbf{G})))=\covol_{\mathbf{G}}(\delta(\Gamma) Z(\mathbf{G}))=
\]
\[
=\covol_{\mathbf{G}}(\delta(\Gamma)) \cdot [\delta(\Gamma) Z(\mathbf{G}) : \delta(\Gamma)]=\covol_{\mathbf{G}}(\delta(\Gamma)) \cdot |Z(\mathbf{G})|,
\]
where the first equality is because the kernel of $\mathbf{h}$ is $Z(\mathbf{G})$, the second is because $\mathbf{h}$ is locally measure preserving, the third is clear, and the forth is because $Z(\mathbf{G}) \cap \delta(\Gamma) \subset \delta(Z(\Gamma))=1$. The same proof shows that
\[
\covol_{\mathbf{G} / Z(\mathbf{G})}(\mathbf{q}(\delta(\rho(\Gamma))))=\covol_{\mathbf{G}}(\delta(\rho(\Gamma))) \cdot |Z(\mathbf{G})|,
\]
which implies the claim.
\end{enumerate}
This complete the proof Theorem \ref{thm:rigidity}.
\end{proof}

\begin{corollary}\label{cor:com.rigid}
Let $\Phi$ be a group which is abstractly commensurable to an irreducible non-uniform higher-rank  lattice. Let $\rho:\Phi \rightarrow \Phi$ be an endomorphism with an infinite image.
Then: 
\begin{enumerate}
\item\label{cor:com.rigid.1} $\ker \rho$ is finite.
\item\label{cor:com.rigid.2} If $\rho$ is injective then $\rho$ is an automorphism. 
\end{enumerate}
\end{corollary}

\begin{proof}  Margulis' Arithmeticity Theorem implies that there exist
connected absolutely simple group $G$ defined over a number field $k$, a finite set of valuations $S$ which  contains all archimedean ones
such that $G(K_v)$ is unbounded for every $v \in S$,
and  a finite-index subgroup $\Gamma$ of $G(O_S)$ such that $\Gamma$ is isomorphic to  a finite index subgroup of $\Phi$. 
Margulis' superrigidity theorem implies that all finite index subgroups of $G(O_S)$ are superrigid in $\prod_{v\in S} G(k_v)$. 
We identify $\Gamma$ with its image in $\Phi$.   There exists a finite index subgroup $\Gamma_1$ of $\Gamma$
such that $\rho(\Gamma_1) \le \Gamma$. 

Since $\Gamma_1$ has a finite index in $\Phi$ then $\rho(\Gamma_1)$ is infinite. Part  \eqref{cond:rigidity.1} of Theorem \ref{thm:rigidity}  implies that $\ker\rho \cap \Gamma_1$ is finite.  Since $\Gamma_1$ has a finite index in $\Phi$,
$\ker\rho$ is also finite. 
 
Assume that $\rho$ is injective. Part \eqref{cond:rigidity.2} of Theorem \ref{thm:rigidity} implies that 
$$[G(O_S):\Gamma][\Gamma:\Gamma_1]=[G(O_S):\Gamma_1]=[G(O_S):\rho(\Gamma_1)]=[G(O_S):\Gamma][\Gamma:\rho(\Gamma_1)].$$
Thus, $[\Gamma:\Gamma_1]=[\Gamma:\rho(\Gamma_1)]$ and $[\Phi:\Gamma_1]=[\Phi:\Gamma][\Gamma:\Gamma_1]=[\Phi:\Gamma][\Gamma:\rho(\Gamma_1)]=[\Phi:\rho(\Gamma_1)]$.
 Since $\rho$ is injective, $[\rho(\Phi):\rho(\Gamma_1)]=[\Phi:\Gamma_1]=[\Phi:\rho(\Gamma_1)]$ and $\rho$ is surjective. 
\end{proof}

\begin{lemma}\label{lem:gen.prop} Let $\Phi$ be a group which is abstractly commensurable to an irreducible non-uniform higher-rank  lattice. Then
\begin{enumerate}
\item\label{lemma:gen.prop.1} $Z(\Phi)$ is finite.
\item\label{lemma:gen.prop.2} $\Phi$ has a maximal finite normal subgroup. 
\item \label{lemma:gen.prop.3} $\Phi$ is finitely presented. 
\item\label{lemma:gen.prop.4} There exists a constant $N$ such that every finite subgroup of $\Phi$ has a normal abelian subgroup of 
index at most $N$. 
\item\label{lemma:gen.prop.5} $\Phi$ contains a non-abelian free subgroup. 
\end{enumerate}
\end{lemma}

\begin{proof} It is well knows that a higher-rank lattice $\Gamma$ has a finite center and that every finite normal subgroup of $\Gamma$ is central. 
Parts \eqref{lemma:gen.prop.1} and  \eqref{lemma:gen.prop.2} easily follows from these facts. For part \eqref{lemma:gen.prop.3}
recall that lattices in semisimple group (of characteristic zero) are finitely presented and that finite presentability  is preserved under 
abstract commensurability. Part \eqref{lemma:gen.prop.4} is just Jordan's theorem about finite linear group of characteristic zero.  
Part \eqref{lemma:gen.prop.5} follows from Tit's alternative \cite{Tit}. 
\end{proof}

We can now prove Theorem \ref{thm:sr.prime}.

\begin{proof}[Proof of Theorem \ref{thm:sr.prime}] By Theorem \ref{thm:prime.criterion}, we need to show that there is a generating set $g_1,\ldots,g_n$ and a formula $\phi(x_1,\ldots,x_n)$ such that, for any tuple $(h_1,\ldots,h_n)\in \Phi^n$, if $\phi(\vec{h})$ holds, then there is an automorphism of $\Phi$ sending $g_i$ to $h_i$ for every $1 \le i \le n$.

We are going to use freely the facts mentioned in Lemma  \ref{lem:gen.prop}. Find a generating tuple $g_1,\ldots g_n$  and let $r_1,\ldots,r_a$ be the corresponding defining relations. Let $w_1,\ldots,w_b$ be words such that $\left\{ w_1(\vec{g}),\ldots,w_b(\vec{g}) \right\}$ is the set of non-trivial elements in the maximal finite normal subgroup of $\Phi$. Let $N$ be the constant defined  in part \eqref{lemma:gen.prop.4}
\label{lemma:gen.prop.4} of Lemma  \ref{lem:gen.prop}.  Since $\Phi$ contains a non-abelian free subgroup, there are words $u_1,u_2$ such that $[u_1(\vec{g})^N,u_2(\vec{g})^N]\neq 1$. Let $\phi(x_1,\ldots,x_n)$ be the formula
\[
\left( [u_1(\vec{x})^N,u_2(\vec{x})^N]\neq 1\right) \wedge \left( \bigwedge_{j=1}^a r_j(\vec{x})=1 \right) \wedge \left( \bigwedge_{i \leq b} w_i(\vec{x}) \neq 1) \right).
\]
Assume that $h_1,\ldots,h_n\in \Phi$ and $\phi(\vec{h})$ holds. There exists an endomorphism $\rho:\Gamma\rightarrow \Gamma$ which sends $g_i$ to $h_i$
for every $1 \le i \le n$. 
Since $[u_1(\vec{h})^N,u_2(\vec{h})^N]\neq 1$, 
the images of the form
$\rho(u_i(\vec{g}))=u_i(\vec{h})$ are not contained in any finite subgroup of $\Gamma$. 
Hence, the image of $\rho$ is infinite. 

Part  \eqref{cor:com.rigid.1} of Corollary \ref{cor:com.rigid} implies that $\ker(\rho)$ is finite, and hence contained in $\{1\}\cup \left\{ w_1(\vec{g}),\ldots,w_b(\vec{g}) \right\}$. 
By the definition of $\phi$, we get that $\rho$ is one-to-one. Part  \eqref{cor:com.rigid.2} of Corollary \ref{cor:com.rigid} implies  that $\rho$ is an automorphism, confirming the required condition.
\end{proof}

\begin{remark} The converse of Theorem \ref{thm:sr.prime}, namely, that a prime lattice is superrigid, is false: by \cite{Sel09}, torsion-free cocompact lattices in $SO(n,1)$, $n \geq 3$ are prime.
It is well known that these lattices are not necessarily (and probably never) superrigid. 
\end{remark}

\begin{remark} Prime  groups need not be first-order rigid. For example, any cocompact lattice in $\Sp(n,1)$ satisfies the assumptions of Theorem \ref{thm:rigidity} (and hence prime) but is not first order rigid by Theorem 7.6 of \cite{Sel09}.
\end{remark}

\begin{remark} The crucial property needed in the proof of Theorem \ref{thm:sr.prime} above is the property stated in Corollary \ref{cor:com.rigid}: Every injective endomorphism of $\Gamma$ with an infinite image is an automorphism. This property does not hold for positive characteristic higher rank lattices.  
For example, if $F$ be a finite a field and $n \ge 3$ then $\SL_n(F[t])$ is supperrigid but it has many proper subgroups that are isomorphic to itself, e.g., $\SL_n(F[t^m])$ for every $m \ge 2$. Nevertheless, we can prove that $\SL_n(F[t])$ and all its finite index subgroups are prime and first order rigid.
\end{remark}

\section{Brenner property for higher-rank groups} \label{sec:Raghunathan}

\begin{theorem} \label{thm:Raghunathan} Let $k$ be a global field and let $S$ be a finite set of valuations, containing all archimedean ones. 
Let $G$ be a simple algebraic group over $k$ which is $k$-isotropic and has $S$-rank at least 2. Let $P$ be a maximal proper $k$-parabolic subgroup, and let $\Gamma$ be a finite index subgroup of $G(O_S)$. There is a constant $C$ such that, for any non-central $\gamma \in \Gamma$, the set $[\gamma ^ \Gamma \cup (\gamma ^{-1})^\Gamma]^C$ contains a finite index subgroup of $U \cap \Gamma$ where $U$ is the unipotent radical of $P$.
\end{theorem} 

\begin{proof} The claim essentially appears in the proof of Theorem 2.1 of \cite{Rag}. We sketch the argument. Fix a maximal $k$-split torus $S$ contained in $P$. There is a simple $k$-root $\alpha$ and an ordering of the simple $k$-roots such that $P$ is the parabolic corresponding to $\alpha$ and the positive roots. By \cite[\S5]{BT-reductive}, there is $w\in N_G(S)(k)$ that switches the positive and negative roots.
The image of $w$ in the Weyl group is of order 2. This means that $w^2\in C_G(S)$, so $P^{w^2}=P$ (because $C_G(S) \subset P$). In particular, $P\cap P^w$ is $w$-invariant. Let $\alpha '=-w(\alpha)$ (so $\alpha'$ is positive), let $P'$ be the (maximal) parabolic corresponding to $\alpha '$, and let $U'$ be its unipotent radical. By \cite[Theorem 5.15]{BT-reductive}, the map $(u,b) \stackrel{\varphi}{\mapsto} uwb$ is a $k$-isomorphism between $U' \times P$ and an open dense set in $G$. By definition, $\varphi^{-1}$ is also defined over $k$, so $uwb\in G(k)$ implies that $u\in U'(k)$ and $b\in P(k)$. 

We first claim that there is a constant $C_1$ such that $[\gamma ^ \Gamma \cup (\gamma ^{-1})^\Gamma]^{C_1}$ is Zariski dense.
Note that $[\gamma ^ \Gamma \cup (\gamma ^{-1})^\Gamma]^2$ contains a Zariski-dense subset of $\gamma ^G \cdot (\gamma ^{-1})^G$ and the later contains the identity. 
The conjugacy class $\gamma ^G$ is irreducible and has positive dimension, and, hence, so is $[\gamma ^G \cdot (\gamma ^{-1})^G]^n$, for all $n \geq 1$. Note also that $n \mapsto \dim [\gamma ^G \cdot (\gamma ^{-1})^G]^n$ is non-decreasing. If $\dim [\gamma ^G \cdot (\gamma ^{-1})^G]^n = \dim [\gamma ^G \cdot (\gamma ^{-1})^G]^{n+1}$, it follows that the Zariski closures of $[\gamma ^G \cdot (\gamma ^{-1})^G]^n$ and $[\gamma ^G \cdot (\gamma ^{-1})^G]^{n+1}$ coincide. Therefore, the Zariski closure of $[\gamma ^G \cdot (\gamma ^{-1})^G]^n$ is a normal subgroup of $G$, so it must be $G$. 

Let $u\in U'$ and $b\in P$ such that $uwb\in [\gamma ^ \Gamma \cup (\gamma ^{-1})^\Gamma]^{C_1}$ and suppose that $x\in \Gamma \cap P \cap P^w$ satisfies $[x,u]\in \Gamma$. 
We consider the effect of conjugating by $x$ and by $[x,u]:=xux^{-1}u^{-1}$ on the Bruhat decomposition of $uwb$:
\[
 (x u x ^{-1}) w (x^w b x^{-1}) = xux ^{-1} xwb x ^{-1}=x (uwb) x ^{-1}  \in [\gamma ^ \Gamma \cup (\gamma ^{-1})^\Gamma] ^{C_1}
\]
and
\[
(x u x ^{-1})  w (b u x u^{-1} x ^{-1}) =[x,u] (uwb) [x,u] ^{-1} \in [\gamma ^ \Gamma \cup (\gamma ^{-1})^\Gamma] ^{C_1}.
\]
Taking the quotient,
\[
x b ^{-1} (x ^{-1})^w b u x u ^{-1} x ^{-1} \in [\gamma ^ \Gamma \cup (\gamma ^{-1})^\Gamma]^{2C_1}
\]
as $[\gamma ^ \Gamma \cup (\gamma ^{-1})^\Gamma]^{2C_1}$ is closed to conjugation by elements of $\Gamma$,
\begin{equation}\label{eq}
b ^{-1} (x ^{-1})^w b u x u ^{-1} \in [\gamma ^ \Gamma \cup (\gamma ^{-1})^\Gamma]^{2C_1}.
\end{equation}
Note that $U'$ is generated by rational positive roots so it is contained in $P$ and in particular $u \in P$.
Since $P\cap P^w$ is $w$-invariant, our assumptions on $x$ imply that every term in equation \eqref{eq} is in $P$.
Thus, the element in \eqref{eq} is also contained in $P$.  Let 
\[
A=\left\{ (u,b,x) \in U' \times P\times (P\cap P^w) \mid uwb\in \Gamma, [x,u]\in \Gamma, x\in \Gamma \right\} 
\]
and let $f: U' \times P\times (P\cap P^w) \rightarrow P$ be the function
\[
f(u,b,x)=b ^{-1} (x ^{-1})^w b u x u ^{-1}.
\]
We just showed that $f(A) \subset [\gamma ^ \Gamma \cup (\gamma ^{-1})^\Gamma]^{2C_1}$. Let $M$ be the connected component of the Zariski closure of $(P \cap P^w)(O_S)$. We claim that $A$ is Zariski dense in $U' \times P\times M$. Indeed, the collection of $(u,b)$ satisfying $uwb\in \Gamma$ is Zariski dense in $U'\times P$, so it is enough to show that, for every $u\in U'(k)$, the collection of $x$'s satisfying $[x,u]\in \Gamma$ contains a finite-index subgroup of $M(O_S)$. After passing to a finite-index subgroup, we can assume that $\Gamma$ is normal in $G(O_S)$. Consider the polynomial function $x \mapsto [x,u]$. It has $k$-rational coefficients and maps $1$ to $1$. It follows that there is an ideal $\mathfrak{a}$ of $O_S$ such that, if $x\in G(O_S)(\mathfrak{a})$, then $[x,u]\in G(O_S)$. Consider the map $c:M(O_S)(\mathfrak{a}) \cap \Gamma \rightarrow G(O_S)/\Gamma$ defined by $c(x)=[x,u]\Gamma$. Since $[xy,u]=xyuy ^{-1} x ^{-1} u ^{-1} = x[y,u]x ^{-1} [x,u]$, it follows that $c$ is a homomorphism. Every element $x$ in $\ker(c)$ (which has finite index in $M(O_S)(\mathfrak{a}) \cap \Gamma$ and hence in $M(O_S)$) satisfies $[x,u]\in \Gamma$, which is what we wanted to prove.

It follows that, in the notation above, $f(A)$ is Zariski dense in $f(U' \times P\times M)$. Hence, the Zariski closure of $[\gamma ^ \Gamma \cup (\gamma ^{-1})^\Gamma]^{2C_1} \cap P$ contains $f(U' \times P\times M)$. Denoting the Levi of $P$ by $L$, \cite[Lemma 2.8]{Rag} says that the group generated by $f(U' \times P\times M)$ contains the identity component of 
the Zariski-closure 
$\overline{L(O_S)}^Z$ of $L(O_S)$.

Let $U^i$, $i=1,\ldots,N$ be the ascending central series of $U$. Each $U^i/U^{i+1}$ is a vector space on which $P$ acts by conjugation. If $v\in (U^i\cap \Gamma)/(U^{i+1} \cap \Gamma)$ and $z\in f(A)$, then $(\Ad(z)-1)v=[v,z]\in [\gamma ^ \Gamma \cup (\gamma ^{-1})^\Gamma]^{4C_1}$. We will use the following simple lemma:

\begin{lemma} Let $k$ be a global field, $O$ its ring of integers, and $S$ a finite set of valuations. For $h_1,\ldots,h_t\in \GL_n(O_S)$ generating a subgroup $H$, the following are equivalent: \begin{enumerate}
\item\label{con.1} $\linspan \left\{ (h-1) k ^n \mid h\in \overline{H}^Z\right\} = k^n$.
\item $\linspan \left\{ (h-1) k ^n \mid h\in H\right\} =k ^n$.
\item There is no $H$-invariant linear functional on $k^n$.
\item $\linspan \left\{ (h_i-1) k ^n \mid 1 \le i \le t\right\} = k ^n$.
\item $(h_1-1) O_S ^n + \cdots + (h_t-1) O_S ^n$ has a finite index in $O_S ^n$.
\end{enumerate} 
\end{lemma}

By \cite[Claim 2.11]{Rag}, $(\overline{L(O_S)}^Z)^0$ acting on $U^i/ U^{i+1}$ satisfies condition \eqref{con.1}. Since the Zariski closure of $\langle f(A) \rangle$ contains $(\overline{L(O_S)}^Z)^0$, the action of $\langle f(A) \rangle$ also satisfies this condition. It follows that there are finitely many elements $h_1,\ldots,h_t\in f(A)$ that satisfy the claim of the lemma. In particular, $[h_1,(U^i\cap \Gamma)/(U^{i+1} \cap \Gamma)]+\ldots+[h_t,(U^i\cap \Gamma)/(U^{i+1} \cap \Gamma)]$ has finite index in $(U^i\cap \Gamma)/(U^{i+1} \cap \Gamma)$. By induction, it follows that $[\gamma ^ \Gamma \cup (\gamma ^{-1})^\Gamma]^{4tNC_1}\cap U$ has finite index in $U(O_S)$ (and, hence, in $U(O_S)\cap \Gamma$).
\end{proof} 

\begin{corollary}\label{corol:BP} Let $\Phi$ be a group which is abstractly commensurable to an irreducible non-uniform higher-rank lattice. There exits $g \in \Phi$ which has the \BP. 
\end{corollary}

\begin{proof} 
Margulis' Arithmeticity Theorem implies that $\Phi$ has a finite index subgroup $\Gamma$ which satisfies the assumptions of Theorem \ref{thm:Raghunathan}. 
We claim that any element $g \in \Gamma \cap U$ of infinite order has the \BP where $U$ is as in the statement of Theorem  \ref{thm:Raghunathan}. 
Indeed, let $D=2ABC$ where $A:=[\Phi:\Gamma]$, $B=|Z(\Gamma)|$ and  $C$ is the constant defined in Theorem  \ref{thm:Raghunathan}. Let $h \in \Phi$ and assume that 
$|[h^\Phi \cup {h^{-1}}^\Phi]^{D}| > D$. Remark \ref{rem:BP} implies that 
$|[h^\Phi \cup {h^{-1}}^\Phi]^{AB}| > AB$ so $[h^\Phi \cup {h^{-1}}^\Phi]^{2AB}$ contains a non-central element of $\Gamma$. 
The definition of $C$ implies that $[h ^ \Gamma \cup (h ^{-1})^\Gamma]^{2ABC}$ contains a finite index subgroup of $\la g\ra$.
Since $g$ has an infinite order,  $[h ^ \Gamma \cup (h ^{-1})^\Gamma]^{2ABC}\cap \langle g \rangle \ne \{1\}$.
\end{proof}

\begin{lemma}\label{lemma:rank} Let $\Phi$ be a group which is abstractly commensurable to an irreducible higher-rank lattice. There exists a constant $D$ such that any finitely generated abelian subgroup of $\Phi$ is generated by at most $D$ elements. 
\end{lemma}
\begin{proof} 
Selberg's lemma implies that any finitely generated linear group of characteristic zero has a torsion free finite index subgroup. Thus, $\Phi$ has a torsion free finite index subgroup which is  an irreducible higher-rank lattice. It is known \cite{Ser} that such lattices have a finite cohomological dimension. Let $C$ be the cohomological dimension of $\Gamma$.  Every finitely generated subgroup of $\Gamma$ has cohomological dimension at most $C$. The cohomological dimension of $\mathbb{Z}^n$ is $n$ so the rank of every finitely generated abelian subgroup of $\Gamma$ is at most $C$. Thus, the minimal number of generators of any abelian subgroup of $\Phi$ is at most $[\Phi:\Gamma]C$. 
\end{proof}

We can now prove Theorem \ref{thm:main}:

\begin{proof}[Proof of Theorem \ref{thm:main}] Theorem \ref{thm:sr.prime}, 
Lemma \ref{lem:gen.prop}, Corollary \ref{corol:BP} and  Lemma \ref{lemma:rank}
show that $\Phi$ satisfies the conditions of Theorem \ref{thm:f.o.r.criterion}. Hence $\Phi$ is first-order rigid.
\end{proof}

\section{First-order rigidity is not commensurability invariant} \label{sec:commensurable}

The goal of this section is to show that first order rigidity is not preserved by finite index subgroups nor by finite extensions. 
The key ingredient is the following theorem of Oger:
\begin{theorem}[\cite{Og91,Og96}]\label{Oger} Let $G$ and $H$ be finite-by-nilpotent groups. Then $G$ and $H$ are elementarily equivalent if and only if 
$G \times \Z$ and $H \times \Z$ are isomorphic.  
\end{theorem}

\begin{remark} \label{rem:ee.finite} The following theorems imply that elementary equivalence classes of nilpotent groups are finite: \begin{enumerate}

\item Baumslag \cite{Bau} proved  that if $A,B,C$ and $D$ are finitely generated group such that $A \times B \cong C \times D$ and $B$ and $D$ have the same finite quotients then 
$A$ and $C$ have the same finite quotients. 
\item Pickel \cite{Pic} proved that if $G$ is a nilpotent group then  the collection of isomorphism classes of nilpotent groups with the same finite quotients as  $G$ is finite.
\end{enumerate}
\end{remark}

We start by giving an example of a first order rigid group which has a finite extension that is not first order rigid. 
The example follows Baumslag's construction \cite{Bau} of non-isomorphic finitely generated groups with the same finite quotients.
Every finitely generated abelian group is first order rigid. In particular, the infinite cyclic group $\Z$ is first order rigid.  
For every coprime $n,m \in \N^+$, let $C_n$ be the cyclic group of order $n$ and let
$\rho_m:\Z \rightarrow \aut(C_n)$ be the homomorphism defined by $\rho_m(1):=\alpha_m$ where
$\alpha_{m}$ is the automorphism of $C_n$ which sends each element to its $m$-th power.
Define $G_6:=C_{25}\mathbb{o}_{\rho_6} \Z $   and $G_{11}:=C_{25}\mathbb{o}_{\rho_{11}} \Z $.
Note that $\Z$ is isomorphic to a finite index subgroup of $G_6$,  so it is enough to prove that $G_6$ is not first order rigid. 

\begin{proposition}$G_6$  is not first order rigid. \end{proposition}
\begin{proof}
For every $i \in \{6,11\}$, the set $\tor(G_i)$ of torsion elements of $G_i$ is a subgroup of $G_i$ which is  isomorphic to $C_{25}$.
If $g \in G$ and $g\tor(G_i)$ generates $G_i/\tor(G_i) \simeq C$ then the conjugation action of $g$ on $\tor(G_i)$  induces either $\alpha_{6}$ or $\alpha_{21}$
if $i=6$ and $\alpha_{11}$ or $\alpha_{16}$ if $i=11$. In particular, $G_6$ and $G_{11}$ are not isomorphic. On the other hand the map
$\psi:G_{11}\times \Z \rightarrow G_{6}\times \Z$ defined by $\psi(((r,s),t))=((r,2s+5t),s+2t)$ is an isomorphism. Theorem \ref{Oger} implies that 
$G_6$ is not first order rigid. 
\end{proof}
Our next goal it to show that $G_6$ has a finite extension which is first order rigid.

\begin{lemma}\label{now lem}
Let $M$ be a finite group such that all the automorphisms of $M$ are inner and define $H_1:=M \times \Z$. 
Then $H_1$ is first order rigid.
\end{lemma}

\begin{proof}
 Since $H_1$ is finite-by-nilpotent every group which is elementarily equivalent to $H_1$
is finite-by-nilpotent. Thus, Theorem \ref{Oger} implies that it is enough to show that if $H_1 \times \Z \cong H_2 \times \Z$ then 
$H_1 \cong H_2$.  Choose an isomorphism $\iota: H_1 \times \Z \rightarrow H_2 \times \Z$.  Note that $M$ and thus $\iota(M)$
are the torsion subgroups of  $H_1 \times \Z$ and $H_2 \times \Z$. In particular, $\iota(M) \le H_2$. 
Thus, $$\Z \times \Z \cong (H_1/M) \times \Z \cong (H_1 \times \Z)/M \cong (H_2 \times \Z)/\iota(M)=
(H_2/\iota(M)) \times \Z.$$
The structure theorem of finitely generated abelian groups implies that $\Z \cong H_1/M \cong H_2/\iota(M)$. 
Thus, $H_2 \cong M \mathbb{o}_\delta \Z$ for some homomorphism $\delta:\Z \rightarrow \aut(M)$. Since all the automorphisms of $M$ are inner,
$H_2 \cong M \times \Z \cong H_1$. 
\end{proof}

\begin{proposition} $G_6$ is a finite index subgroup of a first order rigid group.
\end{proposition}
\begin{proof}
Embed $C_{25} \mathbb{o} \aut(C_{25})$ in the symmetric group  $S_n$ for some $n \ge 7$ and recall that all the automorphism of 
$S_n$ are inner for $n \ne 6$. Every automorphism of $C_{25}$ is the restriction of an inner automorphism of $C_{25} \mathbb{o} \aut(C_{25})$, in particular,  
every automorphism of $C_{25}$ is the restriction of an inner automorphism of $S_n$. 
Define $H:=S_n \mathbb{o}_\gamma \Z$ where $\gamma:\Z\rightarrow \aut(S_n)$ is a homomorphism for which 
 $\gamma(1)\in \aut(S_n)$ is an automorphism which preserves $C_{25}$ and acts on it as 
$\alpha_6$. Then, 
$G_6$ can be identified as a finite index subgroup of $H$.  
Since all the automorphisms of $S_n$ are inner, $H \cong S_n \times \Z$. 
Lemma \ref{now lem} implies that 
$H$ is first order rigid. 
\end{proof}

\end{document}